\documentclass[11pt]{amsart}
\usepackage[english]{babel}
\usepackage[T1]{fontenc}
\usepackage{geometry}

\usepackage{fullpage}
\usepackage[dvips]{graphicx}
\usepackage{color}

\usepackage{amsmath,amscd,amsthm,amsfonts,latexsym,epsfig}

\theoremstyle{plain}
\newtheorem{theorem}                 {Theorem}      [section]
\newtheorem{proposition}  [theorem]  {Proposition}
\newtheorem{corollary}    [theorem]  {Corollary}

\newtheorem*{theorem*}{Theorem}
\theoremstyle{definition}
\newtheorem{example}      [theorem]  {Example}
\newtheorem{remark}       [theorem]  {Remark}
\newtheorem{definition}   [theorem]  {Definition}

\def \r{\mbox{${\mathbb R}$}}

\def \c{\mbox{${\mathbb C}$}}
\def \h{\mbox{${\mathbb H}$}}

\def \s{\mbox{${\mathbb S}$}}

\def \g{\mbox{$g_{\varepsilon}$}}
\def \n{\mbox{${{\nabla}}^\varepsilon$}}

\DeclareMathOperator{\R}{R^{\varepsilon}}

\DeclareMathOperator{\trace}{trace}
\DeclareMathOperator{\cst}{constant}

\begin{document}

\title{Constant angle surfaces in Lorentzian Berger spheres}

\author{Irene I. Onnis}
\address{Departamento de Matem\'{a}tica\\ ICMC/USP-Campus de S\~ao Carlos\\
Caixa Postal 668\\ 13560-970 S\~ao Carlos, SP, Brazil}
\email{onnis@icmc.usp.br}

\author{Apoena Passos Passamani}
\address{Departamento de Matem\'{a}tica, UFES, 29075-910 Vit\'{o}ria, ES, Brazil}
	\email{apoenapp@gmail.com.br}
	
	\author{Paola Piu}\address{Universit\`a degli Studi di Cagliari\\
Dipartimento di Matematica e Informatica\\
Via Ospedale 72\\
09124 Cagliari}
\email{piu@unica.it}

\thanks{The third author was supported by PRIN 2015 ``Variet\`a reali e complesse: geometria, topologia e analisi armonica''  Italy; and GNSAGA-INdAM, Italy.}

\begin{abstract}
In this work, we study helix spacelike and timelike surfaces in the Lorentzian Berger sphere $\s_{\varepsilon}^3$, that is the three-dimensional sphere endowed with a $1$-parameter family of Lorentzian metrics, obtained by deforming the round metric on $\s^3$ along the fibers of the Hopf fibration $\s^3\to \s^2({1}/{2})$ by $-\varepsilon^2$. Our main result provides a characterization of the helix surfaces in $\s_{\varepsilon}^3$ using the symmetries of the ambient space and a general helix in $\s_{\varepsilon}^3$, with axis the infinitesimal generator of the Hopf fibers. Also, we construct some explicit examples of helix surfaces in  $\s_{\varepsilon}^3$.
\end{abstract}

\maketitle

\section{Introduction}

By definition, a {\em helix surface} or {\em constant angle surface}  is a surface whose unit normal vector field forms a constant angle with a fixed field of directions of the ambient space. The study of these surfaces starts with \cite{CDS07}, where are analyzed such surfaces in $\r^3$ obtaining a remarkable relation with a Hamilton-Jacobi equation and showing their application to  equilibrium configurations of liquid crystals. \\

In recent years much work has been done to understand the geometry of the helix surfaces and they have been classsified in all the $3$-dimensional Riemannian geometries (see \cite{DFVV07,DM09,Di,FMV11,LM11,MO,MOP}). Moreover, we remark that helix submanifolds have been studied in higher dimensional euclidean spaces and product spaces (see \cite{DSRH09,DSRH10,RH11}). \\

Concerning the study of helix surfaces in Lorentzian $3$-manifolds, we refer \cite{FN,LM} and \cite{OP}. In \cite{LM}, the authors classified  constant angle spacelike surfaces in the Lorentz-Minkowski $3$-space, while in \cite{FN} are considered  constant angle spacelike and timelike  surfaces in the Lorentzian product spaces given by $\s^2\times \r_1$ and $\h^2 \times \r_1$.   Moreover, in \cite {OP} is given an explicit local parametrization of constant angle  (spacelike and timelike)  surfaces in the three-dimensional Heisenberg group, equipped with a $1$-parameter family of Lorentzian metrics. \\ 

In this paper, we characterize the surfaces in the  Lorentzian Berger sphere $\s^3_\varepsilon$ whose unit normal vector field makes a constant angle with the unit Hopf vector field. We remember that Hopf vector fields on the $3$-sphere are tangent to the fibers of the Hopf fibration $\psi:\s^3\to\s^2({1}/{2})$. When both
manifolds are endowed with their usual metrics, this map is a Riemannian submersion with
totally geodesic fibers, whose tangent space is generated by the vector field $X_1(z,w) = J_1(z,w)$,
where $(z,w)\in\s^3$ and $J_1$ is the usual complex structure of $\r^4$.
\\

The Lorentzian Berger sphere  $\s^3_\varepsilon$ is the usual $3$-sphere equipped with a $1$-parameter family of Lorentzian metrics $g_\varepsilon$, $\varepsilon\neq 0$, that are obtained by deforming  the canonical metric $\langle\;,\;\rangle$ on the sphere $\s^3$ along the fibers of the Hopf fibration $\psi$ in the following way:
$$g_\varepsilon|_{X_1^\perp}=\langle\;,\;\rangle|_{X_1^\perp},\qquad g_\varepsilon(X_1,X_1)=-\varepsilon^2\,\langle X_1,X_1\rangle, \qquad g_\varepsilon(X_1,X_1^\perp)=0.$$
With respect to the metric $g_\varepsilon$, the Hopf vector field $E_1=\varepsilon^{-1}\,X_1$ is a unit Killing vector field and it satisfies the geometric identity:
\begin{equation}\label{eqprima}
\n_ X E_1=-\varepsilon \, X\wedge E_1, \qquad X\in \mathfrak{X}\big(\s^3_{\varepsilon}\big),
\end{equation}
where $\wedge$ is the cross product in $\s^3_{\varepsilon}$ and $\n$ the Levi-Civita connection of  $\s^3_\varepsilon$. We point out that, starting from the equation~\eqref{eqprima} we derive two additional equations (see \eqref{nablaT} and \eqref{X(nu)}) that will be used to determine the shape operator and the Levi-Civita connection of a constant angle surface in $\s^3_\varepsilon$ (see Proposition~\ref{princ}).\\

Our first result towards the classification of the constant angle surfaces in $\s^3_\varepsilon$ is the Proposition~\ref{position-vector}, showing that we can choose local coordinates $(u,v)$ on a helix surface, so that its position vector $F(u,v)$ in the Euclidean space $\r^4$ must satisfy the differential equation:
$$\frac{\partial^4F}{\partial u^4}+(\tilde{b}^2-2\tilde{a})\,\frac{\partial^2F}{\partial u^2}+\tilde{a}^2\,F=0\,,$$
where
$\tilde{a}$ and $\tilde{b}$ are real constants depending on $\varepsilon$ and $\nu:=g_\varepsilon(N,E_1)\,g_\varepsilon(N,N)$.  Here, $N$ is the unit normal to the helix surface. Moreover, in Proposition~\ref{pro-viceversa} are given necessary and sufficient conditions that an immersion must fulfill in order to define a helix surface in $\s^3_\varepsilon$.\\

Combining these two propositions, we prove the main result of this paper, the Theorem~\ref{teo-principal}, that provides an explicit local description of the semi-Riemannian helix surfaces with constant angle function $\nu$ in $\s^3_\varepsilon$, by means of 
 a suitable $1$-parameter family of isometries of the ambient space and a geodesic of a $2$-torus in the $3$-dimensional sphere. 
Moreover, we investigate the properties of this curve, showing that it is a general helix in $\s^3_\varepsilon$ with axis $E_1$ and, also, that the hyperbolic angle of the general helix is equal to the hyperbolic angle between the unit normal to the helix surface and $E_1$.

%
%
%
 
\section{Preliminaries}
The $3$-dimensional {\it Lorentzian Berger sphere} is defined, using the Hopf fibration, as follows.  
Let
$\s^2({1}/{2})=\{(z,t)\in\c\times\r\colon |z|^2+t^2={1}/{4}\}$ be the usual  $2$-sphere and let $\s^3=\{(z,w)\in\c^2\colon |z|^2+|w|^2=1\}$ be the usual $3$-sphere. Then the Hopf map $$\psi:\s^3\to\s^2({1}/{2}),$$
given by
$$
\psi(z,w)=\frac{1}{2}\,(2z\bar{w},|z|^2-|w|^2)\,,
$$
is a Riemannian submersion and the vector fields
$$
X_1(z,w)=(iz,iw),\quad X_2(z,w)=(-i\bar{w},i\bar{z}),\quad X_3(z,w)=(-\bar{w},\bar{z})
$$
parallelize $\s^3$, with $X_1$  vertical and $X_2$, $X_3$ horizontal. The Lorentzian Berger sphere $\s^3_\varepsilon$, $\varepsilon>0$, is the sphere $\s^3$ endowed with the $1$-parameter family of Lorentzian metrics given by:
$$
g_\varepsilon(X,Y)=\langle X,Y\rangle-(\varepsilon^2+1)\langle X,X_1\rangle \,\langle Y,X_1\rangle,
$$
where $\langle,\rangle$ represents the canonical metric of $\s^3$. \\

Considering the orthonormal basis of $\s^3_{\varepsilon}$ defined by
\begin{equation}\label{eq-basis}
    E_1=\varepsilon^{-1}\,X_1,\quad
     E_2=X_2,\quad
     E_3=X_3,
\end{equation}
 the Levi-Civita connection $\n$ of $\s^3_{\varepsilon}$ is given by:
\begin{equation}
\begin{array}{lll}
\n_{E_{1}}E_{1}=0,& \n_{E_{2}}E_{2}=0,& \n_{E_{3}}E_{3}=0,\\
\n_{E_{1}}E_{2}=\varepsilon^{-1}(2+\varepsilon^2)E_{3}, & &\n_{E_{1}}E_{3}=-\varepsilon^{-1}(2+\varepsilon^2)E_{2},\\
 \n_{E_{2}}E_{1}=\varepsilon E_{3},& \n_{E_{3}}E_{1}=-\varepsilon E_{2},&
 \n_{E_{3}}E_{2}=-\varepsilon E_{1}=-\n_{E_{2}}E_{3}.
\end{array}
\label{nabla}
\end{equation}
The (timelike) unit Killing vector field $E_1$, called the {\em Hopf vector field}, is  tangent to the fibers of the submersion $\psi$ and it satisfies the following identity:
\begin{equation}\label{nablaE1}
\n_ X E_1=-\varepsilon \, X\wedge E_1, \qquad X\in \mathfrak{X}\big(\s^3_{\varepsilon}\big),
\end{equation}
where $\wedge$ is the cross product in $\s^3_{\varepsilon}$, that is defined by the formula: 
$$U\wedge V=-(u_2\,v_3-u_3\, v_2)\,E_1+(u_3\,v_1-u_1\,v_3)\,E_2+(u_1v_2-u_2v_1)\,E_3.$$

For the Riemann curvature tensor, we adopt the convention 
$$\R(X,Y)Z=\n_X\n_Y Z-\n_Y\n_X Z-\n_{[X,Y]}Z,$$
which confers the following non null components:
\begin{equation}\label{R1}
\begin{aligned}
\R(E_1,E_2)E_1&=-\varepsilon^2\, E_2\, , \qquad &\R(E_1,E_3)E_1&=-\varepsilon^2\, E_3\, ,\\ 
\R(E_1,E_2)E_2&=-\varepsilon^2\, E_1\, , \qquad & \R(E_1,E_3)E_3&=-\varepsilon^2\, E_1\, ,  \\
  \R(E_2,E_3)E_3&= (4+3\varepsilon^2)E_2\, ,  \qquad &  \R(E_2,E_3)E_2&= -(4+3\varepsilon^2)E_3\, .
\end{aligned}
\end{equation}
Consequently, we have the following result.
\begin{proposition}\label{curv}
	The Riemann curvature tensor $\R$ of $\s^3_{\varepsilon}$ is determined by
	\begin{equation}\label{R2}
	\begin{aligned}
	\R(X,Y)Z=&(4+3\varepsilon^2)\big[g_{\varepsilon}(Y,Z)X-g_{\varepsilon}(X,Z)Y\big]\\
	&+4(1+\varepsilon^2)\big[g_{\varepsilon}(Y,E_1)g_{\varepsilon}(Z,E_1)X-g_{\varepsilon}(X,E_1)g_{\varepsilon}(Z,E_1)Y\\
	&+g_{\varepsilon}(X,E_1)g_{\varepsilon}(Y,Z)\,E_1-g_{\varepsilon}(Y,E_1)g_{\varepsilon}(X,Z)\,E_1\big],
	\end{aligned}
	\end{equation} 
	for all vector fields $X, Y, Z$ on $\s^3_{\varepsilon}$.
\end{proposition}
\begin{proof}
	Firstly, we decompose the vectors $X,Y,Z$ as
$$X=\overline{X}+x\,E_1, \qquad Y=\overline{Y}+y\,E_1, \qquad Z=\overline{Z}+z\,E_1,$$
where $\overline{X},\overline{Y},\overline{Z}$ are orthogonal to $E_1$ and $x=-g_{\varepsilon}(X,E_1)$, etc. Now, using \eqref{R1} and the properties of the Riemann curvature tensor, we conclude that the terms  
$$\g\big(\R(X,Y)Z,W\big)$$ 
where $E_1$ appears one, three or four times are null. So, for every vector field $W$ in $\s^3_{\varepsilon}$, we have 
$$\begin{aligned}\g\big(\R(X,Y)Z,W\big)=&g\big(\R(\overline{X},\overline{Y})\overline{Z},\overline{W}\big)\\
&+y\,z\,\g\big(\R(\overline{X},E_1)E_1,\overline{W}\big)+x\,z\,\g\big(\R(E_1,\overline{Y})E_1,\overline{W}\big)\\
&+w\,x\,\g\big(\R(E_1,\overline{Y})\overline{Z},E_1\big)+y\,w\,\g\big(\R(\overline{X},E_1)\overline{Z},E_1\big).
\end{aligned}$$
Using \eqref{R1}, it is easy to see that 
$$g\big(\R(\overline{X},\overline{Y})\overline{Z},\overline{W}\big)=(4+3\varepsilon^2)\big[g_{\varepsilon}(\overline{Y},\overline{Z})\g(\overline{X},\overline{W})-g_{\varepsilon}(\overline{X},\overline{Z})\g(\overline{Y},\overline{W})\big]$$
and 
$$\R(\overline{X},E_1)E_1=\varepsilon^2\, \overline{X}, \qquad \R(E_1,\overline{Y})E_1=-\varepsilon^2\, \overline{Y}.$$
Therefore, we get
$$\begin{aligned}\g\big(\R(X,Y)Z,W\big)=&(4+3\varepsilon^2)\big[g_{\varepsilon}(\overline{Y},\overline{Z})\g(\overline{X},\overline{W})-g_{\varepsilon}(\overline{X},\overline{Z})\g(\overline{Y},\overline{W})\big]\\
&+\varepsilon^2\big[y\,z\,\g(\overline{X},\overline{W})-x\,z\,\g(\overline{Y},\overline{W})+w\,x\,\g(\overline{Y},\overline{Z})-y\,w\,\g(\overline{X},\overline{Z})\big]\\
=&\g\big((4+3\varepsilon^2)\big[g_{\varepsilon}(Y,Z)X-g_{\varepsilon}(X,Z)Y\big]\\
	&+ 4(1+\varepsilon^2)\big[g_{\varepsilon}(Y,E_1)g_{\varepsilon}(Z,E_1)X-g_{\varepsilon}(X,E_1)g_{\varepsilon}(Z,E_1)Y\\
	&+g_{\varepsilon}(X,E_1)g_{\varepsilon}(Y,Z)\,E_1-g_{\varepsilon}(Y,E_1)g_{\varepsilon}(X,Z)\,E_1\big], \, W\big).
\end{aligned}$$
 As $W$ is arbitrary, we obtain the equation~\eqref{R2}
\end{proof}

We finish this section, recalling that
the isometry group of $\s^3_{\varepsilon}$ can be identified with:
$$
\{Q\in \mathrm{O}(4)\colon Q J_{1}=\pm  J_{1}Q\}\,,
$$
where $ J_{1}$ is the complex structure of $\r^4$ defined by
$$
 J_{1} =\begin{pmatrix}
 0&-1 & 0 & 0 \\
 1&0 & 0 & 0 \\
 0&0 & 0 & -1 \\
  0&0 & 1 & 0 \\
 \end{pmatrix},
 $$
while $\mathrm{O}(4)$ is the orthogonal group.  In \cite{MO}, Montaldo and Onnis describe explicitly a $1$-parameter family $Q(v)$ of $4\times 4$ orthogonal matrices commuting (respectively, anticommuting) with $J_1$ 
 by using four functions $\xi_1,\xi_2,\xi_3$ and $\xi$ as:
\begin{equation}\label{eq-descrizione-A}
Q(\xi,\xi_1,\xi_2,\xi_3)(v)=
\begin{pmatrix}
{\mathbf r}_1(v)\\
\pm J_1{\mathbf r}_1(v)\\
\cos\xi(v) J_2 {\mathbf r}_1(v)+\sin\xi(v) J_3 {\mathbf r}_1(v)\\
\mp\cos\xi(v) J_3 {\mathbf r}_1(v)\pm\sin\xi(v) J_2 {\mathbf r}_1(v)
\end{pmatrix}\,,
\end{equation}
where
$$
{\mathbf r}_1(v)=(\cos\xi_1(v)\cos\xi_2(v), -\cos\xi_1(v)\sin\xi_2(v), \sin\xi_1(v)\cos\xi_3(v),-\sin\xi_1(v)\sin\xi_3(v))
$$
and 
$$
 J_{2} =\begin{pmatrix}
 0&0 & 0 & -1 \\
 0&0 & -1 & 0 \\
 0&1& 0 & 0 \\
  1&0 & 0 & 0 \\
 \end{pmatrix}\,,\qquad
 J_{3} =\begin{pmatrix}
 0&0 & -1 &0 \\
 0&0 & 0 & 1 \\
 1&0& 0 & 0 \\
0&-1 & 0 & 0 \\
 \end{pmatrix}\,.
 $$
\section{The structure equations for surfaces in $\s^3_{\varepsilon}$}
In this section, we determine the Gauss and Codazzi equations for an oriented pseudo-Riemannian  surface $M$ immersed into $\s^3_{\varepsilon}$. In particular, in the Proposition~\ref{GC} we will prove that these equations involve the metric of $M$, its
shape operator $A$, the tangential projection $T$ of the Hopf vector field $E_1$ and the {\it angle function} $\nu:=g_\varepsilon(N,E_1)\,g_\varepsilon(N,N)$, where $N$ is the unit normal to $M$. \\

First of all, we remember that the surface $M$ is called {\it spacelike} if the induced metric on $M$ by the immersion is Riemannian, and {\it timelike} if the induced metric is Lorentzian. Also, $\g(N,N)=\lambda$, where $\lambda=-1$ if $M$ is a spacelike surface, and $\lambda=1$ if $M$ is timelike.\\

The Gauss and Weingarten formulas, for all $X,Y\in C(TM)$, are  
\begin{equation}
\begin{aligned}
\n_X Y&=\nabla_X Y+ \alpha(X,Y),\\
\n_X N&=-A(X),
\end{aligned}
\end{equation}
where $\nabla$ is the Levi-Civita connection on $M$ and $\alpha$ the second fundamental form with respect to the immersion. In this way, we have 
$$\alpha(X,Y)=\lambda\,\g\big(A(X),Y\big)\,N, \qquad X,Y\in C(TM).$$
Projecting the vector field $E_1$ onto $TM$ we obtain
$$E_1=T+\nu N,$$
where $\nu:M\to\r$ is the angle function. The tangent part of $E_1$, the vector field $T$, satisfies 
\begin{equation}\label{gT}
\g(T,T)=-(1+\lambda \, \nu^2).
\end{equation}
Also, with respect to this decomposition of $E_1$, for all $X\in C(TM)$, we have
\begin{equation*}
\begin{aligned}
\n_X E_1=&\n_X T+X(\nu)\, N+\nu\, \n_X N\\
=&\nabla_X T+X(\nu)\, N+\lambda\, \g\big(A(X),T\big)\, N-\nu \, A(X).
\end{aligned}
\end{equation*}
On the other hand, equation~\eqref{nablaE1} gives
$$\n_X E_1=-\varepsilon\, \lambda \, \g(JX,T)\, N+\varepsilon\, \nu \, JX,$$
where $JX:=N\wedge X$ satisfies
\begin{equation}\label{JX}
\g\big(JX,JY\big)=-\lambda\,\g(X,Y), \qquad J^2X=\lambda\, X.
\end{equation}
Then, comparing the tangent and normal components, we obtain the following equations:
\begin{equation}\label{nablaT}
\nabla_X T=\nu\big(A(X)+\varepsilon\, JX\big)
\end{equation}
and 
\begin{equation}\label{X(nu)}
X(\nu)=-\lambda \, \g\big(A(X)+\varepsilon\, JX,T\big).
\end{equation}

Now, we will give the expressions of the Gauss and Codazzi equations for a pseudo-Riemannian surface $M$ immersed into $\s^3_{\varepsilon}$.
\begin{proposition}\label{GC}
Under the above notation, the Gauss and Codazzi equations in $\s^3_{\varepsilon}$ are given, respectively, by:
\begin{equation}\label{eq-gauss}
K=\overline{K}+\lambda\,\mathrm{det} A=-\varepsilon^2+\lambda\,[\mathrm{det} A-4\,\nu^2\,(1+\varepsilon^2)]
    \end{equation}
and
    \begin{equation}\label{eq-codazzi}
\nabla_X A(Y)-\nabla_Y A(X)-A[X,Y]=4\,\lambda\,\nu\,(1+\varepsilon^2)\,[g_\varepsilon(X,T)Y-g_\varepsilon(Y,T)X],
    \end{equation}
    where $X$ and $Y$ are tangent vector fields on $M$, $K$ is the Gauss curvature of $M$ and $\overline{K}$ denotes the sectional curvature in $\s^3_{\varepsilon}$ of the plane tangent to $M$. 
    \end{proposition}
\begin{proof}
Firstly, we prove the equation~\eqref{eq-gauss}. Recall that the Gauss equation for a pseudo-Riemannian hypersurface takes the form:
 \begin{equation}\label{G1}
     K=\overline{K}+\lambda\,\frac{\g(A(X),X)\g(A(Y),Y)-\g(A(X),Y)^2}{\g(X,X)\g(Y,Y)-\g(X,Y)^2},
    \end{equation}
where $\overline{K}$ denotes the sectional curvature in $\s^3_{\varepsilon}$ of the tangent plane to $M$.
Also,  supposing that $\{X,Y\}$ is a local orthonormal frame on $M$, i.e. $\g(X,X)=1$, $\g(X,Y)=~0$, $\g(Y,Y)=-\lambda$, the Proposition~\eqref{curv} gives 
\begin{equation*}
\begin{aligned}
\overline{K}(X,Y)=&-\lambda\,\g\big(\R(X,Y)Y,X\big)\\
=& -\varepsilon^2-4 \lambda\,\nu^2\,(1+\varepsilon^2).
\end{aligned}
\end{equation*}
Since $X$ and $Y$ are orthonormal, we have that
$$\mathrm{det} A=\frac{\g(A(X),X)\g(A(Y),Y)-\g(A(X),Y)^2}{\g(X,X)\g(Y,Y)-\g(X,Y)^2}.$$
Combining the above expressions, we obtain \eqref{eq-gauss}.\\

As regards equation \eqref{eq-codazzi}, we consider the Codazzi equation for hypersurfaces, that is given by:
    $$\g(\R(X,Y)Z,N)=\g(\nabla_X A(Y)-\nabla_Y A(X)-A[X,Y],Z).$$
Since Proposition~\ref{curv} implies that
$$\R(X,Y)N=-4\,\lambda\,\nu\,(1+\varepsilon^2)\,[g_\varepsilon(X,T)Y-g_\varepsilon(Y,T)X],$$
the result follows from the arbitrariness of $Z$.
\end{proof} 

\section{Constant angle surfaces in $\s^3_{\varepsilon}$}	
We start this section giving the following definition:
\begin{definition}
	Let $M$ be an oriented pseudo-Riemannian surface in the Lorentzian Berger sphere $\s^3_{\varepsilon}$ and let $N$ be a unit normal vector field, with $\g(N,N)=\lambda$. We say that $M$ is an {\it helix surface} or {\it constant angle surface} if the angle function $\nu:=\lambda\,\g(N,E_1)$ is constant at every point of the surface.
\end{definition}

 We observe that if $M$ is a constant angle spacelike surface, then $|\nu|> 1$. In fact, if $|\nu|=1$, then the vector fields $E_2$ and $E_3$  would be tangent to the surface $M$, which is absurd since the horizontal distribution of the Hopf map $\psi$ is not integrable. Moreover, we note that if $M$ is a timelike surface with $\nu=0$, we have that $E_1$ is always tangent to $M$ and, therefore, $M$ is a Hopf tube. Consequently, from now on,  for a helix timelike surface $M$ we will assume that the constant $\nu\neq 0$.

\begin{proposition}\label{princ}
	Let $M$ be an oriented helix surface with constant angle function $\nu$  in $\s^3_\varepsilon$ and let $N$ be its unit normal vector field, with $\g(N,N)=\lambda$. Then, we have that:
	\begin{itemize}
		\item[(i)] with respect to the basis $\{T,JT\}$, the matrix associated to the shape operator $A$ takes the following form:
		$$
		A=\left(
		\begin{array}{cc}
		0 & -\lambda\,\varepsilon \\
		\varepsilon & \mu\\
		\end{array}
		\right)\,,
		$$
		for some smooth function $\mu$ on $M$;
		\item[(ii)] the Levi-Civita connection $\nabla$ of $M$ is given by:
		$$
		\nabla_T T= 2\varepsilon\,\nu\, JT,\qquad \nabla_{JT} T=\mu\,\nu\, JT\,,
		$$
		$$
		\nabla_T JT=2\lambda\,\varepsilon\,\nu\, T,\qquad \nabla_{JT} JT=\lambda\,\mu\,\nu\, T\,;
		$$
		\item[(iii)] the Gauss curvature of $M$ is constant and it satisfies 
	\begin{equation}\label{cur-gauss}
		K=-4\,\lambda\,(1+\varepsilon^2)\,\nu^2;
	\end{equation}
		\item[(iv)] the function $\mu$ satisfies the following equation
		\begin{equation}\label{lambda}
		T(\mu)+\nu\,\mu^2+4\,\lambda\,\nu\,B=0, 
		\end{equation}
		where the constant
\begin{equation}\label{eq-defB}
B:=1+\lambda\,\nu^2\,(1+\varepsilon^2).
\end{equation}
	\end{itemize}
\end{proposition}
\begin{proof} We start observing that if $M$ is spacelike (respectively, timelike), then $T$ is spacelike (respectively, timelike) and $JT$ is spacelike. 
	Also, from \eqref{gT} and \eqref{JX} we get 
	$$
	\g( T,T)=-(1+\lambda \, \nu^2), \qquad \g( JT,JT)=\lambda+\nu^2>0,\qquad \g( T,JT)=0\,.
	$$ 
	Then, from \eqref{JX} and \eqref{X(nu)}, we obtain that:
	\begin{equation*}
\begin{aligned}
\g(A(T),T)&=0,\\
\g(A(JT),T)&=\varepsilon\,(\lambda+\nu^2)
\end{aligned}
\end{equation*}
	 and, therefore, we have the expression of the matrix $A$ given in (i).  The Levi-Civita connection of $M$ is determined using \eqref{nablaT} and (i). Also, taking into account (i), from \eqref{eq-gauss} we obtain the Gauss curvature of $M$ as in \eqref{cur-gauss}.\\
	 
	Finally, \eqref{lambda} follows from the Codazzi equation~\eqref{eq-codazzi} by putting $X=T$, $Y=JT$ and using (ii). In fact, it is easy to verify that
 $$4\,\lambda\,\nu\,(1+\varepsilon^2)\,[g_\varepsilon(T,T)JT-g_\varepsilon(JT,T)T]=-4\,\nu\,(1+\varepsilon^2)\,(\lambda+\nu^2)$$ and
 $$\begin{aligned}&\nabla_T A(JT)-\nabla_{JT} A(T)-A[T,JT]\\&=
 \nabla_T (-\lambda\,\varepsilon\, T+\mu\, JT)-\nabla_{JT} (\varepsilon\, JT)-\big[\nu\,(2\lambda\,\varepsilon^2-\mu^2)\,JT+\varepsilon\,\lambda\,\mu\,\nu\, T\big]\\&=
 [T(\mu)+\nu\,\mu^2-4\lambda\,\epsilon^2\,\nu]\,JT.
 \end{aligned}$$
	\end{proof}

	\begin{remark}\label{value-B}
We observe that if $M$ is a spacelike (respectively, timelike) surface, then the constant $B$ is negative (respectively, positive). Therefore, in both cases we have that $\lambda\,B$ is positive.
Consequently, if a helix surface is minimal (i.e. $\trace A=0$), 
from (i) of the Proposition~\ref{princ} it follows that $\mu=0$ and, so, using \eqref{lambda} we get $\lambda\,\nu\,B=0$. Thus, $\nu=0$ and the surface is a timelike Hopf tube.
\end{remark}

As $g_\varepsilon (E_1,N)=\lambda\,\nu$ and $E_1$ is a timelike vector field, there exists a smooth function $\varphi$ on $M$ so that
$$N=-\lambda\,\nu\, E_1+\sqrt{\lambda+\nu^2}\cos\varphi\,E_2+\sqrt{\lambda+\nu^2}\sin\varphi\,E_3.$$
Therefore, 
\begin{equation}\label{eq:def-T}
T=E_1-\nu\,N=(1+\lambda\,\nu^2)\,E_1-\nu\sqrt{\lambda+\nu^2}\cos\varphi\,E_2-\nu\sqrt{\lambda+\nu^2}\sin\varphi\,E_3
\end{equation}
and 
$$
JT=\sqrt{\lambda+\nu^2}\,(\sin\varphi\,E_2-\cos\varphi\,E_3)\,.
$$
In addition,
\begin{equation}\begin{aligned}\label{eqTJ}
 A(T)&=-\n_T N=[T(\varphi)+\varepsilon^{-1}(2+\varepsilon^2)\,(1+\lambda\,\nu^2)+\lambda\,\varepsilon\,\nu^2]\,JT,\\
A(JT)&=-\n_{JT} N=JT(\varphi)\,JT-\lambda\,\varepsilon\,T\,.
\end{aligned}
\end{equation}
Comparing \eqref{eqTJ} with (i) of  Proposition~\ref{princ}, we have that
\begin{equation}\label{eqTJ1}
\left\{\begin{aligned}
JT(\varphi)&=\mu\,,\\
T(\varphi)&=-2\varepsilon^{-1}\,B\,.
\end{aligned}
\right.
\end{equation}
We point out that, as $$[T,JT]=\nu\,(2\lambda\,\varepsilon\, T-\mu\,JT),$$ the compatibility condition of system~\eqref{eqTJ1}:
$$(\nabla_T JT-\nabla_{JT} T)(\varphi)=[T,JT](\varphi)=T(JT(\varphi))-JT(T(\varphi))$$ is equivalent to \eqref{lambda}.\\

We now choose local coordinates $(u,v)$ on $M$ such that 
\begin{equation}\label{eq:local-coordinates}
\partial_u=T, \qquad \partial_v=a\,T+b\,JT,
\end{equation} for certain smooth functions $a=a(u,v)$ and $b=b(u,v)$. As
$$
0=[\partial_u,\partial_v]=(a_u+2\lambda\,\varepsilon\,\nu\, b)\,T+(b_u-\nu\,\mu\,b)\,JT\,,
$$
it results that
\begin{equation}\label{eqab}
\left\{\begin{aligned}
a_u&=-2\lambda\,\varepsilon\,\nu\, b,\\
b_u&=\nu\,\mu\,b\,.
\end{aligned}
\right.
\end{equation}
Also, we can write \eqref{lambda}  as
$$\mu_u+\nu\,\mu^2+4\,\lambda\,\nu\,B=0,$$
where the constant $\lambda\,B$ is  positive (see Remark~\ref{value-B}). So, by integration, we have:
\begin{equation}\label{eqlambda}
\mu(u,v)=2\,\sqrt{\lambda\,B}\tan \big(\eta(v)-2\,\nu\, \sqrt{\lambda\,B}\,u\big)\,,
\end{equation}
for some smooth function $\eta$ depending on $v$ and  we can solve  system~\eqref{eqab}.
As we are interested in only one coordinate system on the surface $M$  we only need one solution for $a$ and $b$, for example:
\begin{equation}\label{solab}
\left\{\begin{aligned}
a(u,v)&=\frac{\lambda\,\varepsilon}{\sqrt{\lambda\, B}}\sin \big(\eta(v)-2\nu\, \sqrt{\lambda \, B}\,u\big),\\
b(u,v)&=\cos\big(\eta(v)-2\nu \,\sqrt{\lambda \, B}\,u\big)\,.
\end{aligned}
\right.
\end{equation}
Therefore \eqref{eqTJ1} becomes
\begin{equation}\label{eqTJ2}\left\{\begin{aligned}
\varphi_u&=-2\varepsilon^{-1}B\,,\\
\varphi_v&=0\,,
\end{aligned}
\right.
\end{equation}
of which the general solution is given by
\begin{equation}
\varphi(u,v)=-2\varepsilon^{-1}\,B\,u+c\,,\qquad c\in\r.
\end{equation}

Using the previous results, we prove the following:
\begin{theorem}\label{position-vector}
	Let $M$ be a helix surface in the Lorentzian Berger sphere $\s^3_\varepsilon$ with constant angle function $\nu$.  Then, with respect to the local coordinates $(u,v)$ on $M$ defined in \eqref{eq:local-coordinates} and \eqref{solab},  the position vector $F$ of $M$ in $\r^4$ satisfies the equation:
	\begin{equation}\label{eqquarta}
	\frac{\partial^4F}{\partial u^4}+(\tilde{b}^2-2\tilde{a})\,\frac{\partial^2F}{\partial u^2}+\tilde{a}^2\,F=0\,,
	\end{equation}
	where
	\begin{equation}\label{eq:value-a-b}
	\tilde{a}=\lambda\,\varepsilon^{-2}\, B\,(\lambda+\nu^2)>0, \qquad \tilde{b}=-2\varepsilon^{-1}\, B
	\end{equation}
	and $B=1+\lambda\, \nu^2\,(1+\varepsilon^2)$.
\end{theorem}
\begin{proof}
	Let $M$ be a helix surface and let $F$ be the position vector of $M$ in $\r^4$. Then, with respect to the local coordinates $(u,v)$ on $M$ defined in \eqref{eq:local-coordinates} and \eqref{solab}, we can write $F(u,v)=(F_1(u,v),\dots,F_4(u,v))$. By definition, taking into account \eqref{eq:def-T}, we have that
	$$
	\begin{aligned}\partial_u F&=(\partial_uF_1,\partial_uF_2,\partial_uF_3,\partial_uF_4)=T\\
	&=\sqrt{\lambda+\nu^2}\,[\lambda\,\sqrt{\lambda+\nu^2}\,{E_1}_{|F(u,v)}-\nu\cos\varphi\,{E_2}_{|F(u,v)}-\nu\sin\varphi\,{E_3}_{|F(u,v)}]\,.
	\end{aligned}
	$$
	Using the expression of $E_1$, $E_2$ and $E_3$ with respect to the coordinates vector fields of $\r^4$, the latter implies that
	\begin{equation}\label{eqprime}\left\{\begin{aligned}
	\partial_uF_1&=\sqrt{\lambda+\nu^2}\,(-\varepsilon^{-1}\lambda\,\sqrt{\lambda+\nu^2}\,F_2+\nu\cos\varphi\,F_4+\nu\sin\varphi\,F_3)\,,\\
	\partial_uF_2&=\sqrt{\lambda+\nu^2}\,(\varepsilon^{-1}\lambda\,\sqrt{\lambda+\nu^2}\,F_1+\nu\cos\varphi\,F_3-\nu\sin\varphi\,F_4)\,,\\
	\partial_uF_3&=-\sqrt{\lambda+\nu^2}\,(\varepsilon^{-1}\lambda\,\sqrt{\lambda+\nu^2}\,F_4+\nu\cos\varphi\,F_2+\nu\sin\varphi\,F_1)\,,\\
	\partial_uF_4&=\sqrt{\lambda+\nu^2}\,(\varepsilon^{-1}\lambda\,\sqrt{\lambda+\nu^2}\,F_3-\nu\cos\varphi\,F_1+\nu\sin\varphi\,F_2)\,.\\
	\end{aligned}
	\right.
	\end{equation}
	Moreover, taking the derivative with respect to $u$ of \eqref{eqprime} we get
\begin{equation}\label{eqsegunda}\left\{\begin{aligned}
(F_1)_{uu}&=\tilde{a}\,F_1+\tilde{b}\,(F_2)_u\,,\\
(F_2)_{uu}&=\tilde{a}\,F_2-\tilde{b}\,(F_1)_u\,,\\
(F_3)_{uu}&=\tilde{a}\,F_3+\tilde{b}\,(F_4)_u\,,\\
(F_4)_{uu}&=\tilde{a}\,F_4-\tilde{b}\,(F_3)_u\,,\\
\end{aligned}
\right.
\end{equation}
where, using \eqref{eqTJ2},
	$$
		\tilde{a}=\langle F_u,F_u\rangle=\lambda\,\varepsilon^{-2}\, B\,(\lambda+\nu^2), \qquad \tilde{b}=-2\varepsilon^{-1}\, B.
	$$ 
	Finally, taking twice the derivative of \eqref{eqsegunda} with respect to $u$ and using  \eqref{eqprime}--\eqref{eqsegunda} we obtain the  equation~\eqref{eqquarta}.
\end{proof}
Integrating \eqref{eqquarta}, we have the following
\begin{corollary}\label{cor-Fuv}
Let $M$ be a helix surface in $\s^3_\varepsilon$ with constant angle function $\nu$.  Then, with respect to the local coordinates $(u,v)$ on $M$ defined in \eqref{eq:local-coordinates} and \eqref{solab},  the position vector $F$ of $M^2$ in $\r^4$
is given by
$$
F(u,v)=\cos(\alpha_1 u)\,g^1(v)+\sin(\alpha_1 u)\,g^2(v)+\cos(\alpha_2 u)\,g^3(v)+\sin(\alpha_2 u)\,g^4(v),
$$
where
$$
\alpha_{1,2}=\frac{1}{\varepsilon}(\lambda\,B\pm\varepsilon\,|\nu|\, \sqrt{\lambda\,B} )
$$  
are real constant, while the $g^i(v)$, $i\in\{1,\dots,4\}$, are mutually orthogonal vectors fields  in $\r^4$, depending only on $v$, such that 
\begin{equation}\label{G11}
g_{11}=\langle g^1(v),g^1(v)\rangle=g_{22}=\langle g^2(v),g^2(v)\rangle=\frac{\lambda\,\varepsilon}{2B} \alpha_2\,,
\end{equation}
\begin{equation}\label{G33}
g_{33}=\langle g^3(v),g^3(v)\rangle=g_{44}=\langle g^4(v),g^4(v)\rangle=\frac{\lambda\,\varepsilon}{2B} \alpha_1\,.
\end{equation}
\end{corollary}
\begin{proof}
First of all, from the Remark~\ref{value-B}, we conclude that 
$$\tilde{b}^2-2\tilde{a}=2\varepsilon^{-2}\,\lambda\,B\,(\lambda+\nu^2+2\varepsilon^2\,\nu^2)>0,\qquad \tilde{b}^2-4\tilde{a}=4\lambda\,B\,\nu^2>0.$$
Also, integrating the equation~\eqref{eqquarta} we obtain 
$$
F(u,v)=\cos(\alpha_1 u)\,g^1(v)+\sin(\alpha_1 u)\,g^2(v)+\cos(\alpha_2 u)\,g^3(v)+\sin(\alpha_2 u)\,g^4(v)\,,
$$
where
$$
\alpha_{1,2}=\sqrt{\frac{\tilde{b}^2-2\tilde{a}\pm\sqrt{\tilde{b}^4-4\tilde{a}\tilde{b}^2}}{2}}
$$ 
are two constants, while the $g^i(v)$, $i\in\{1,\dots,4\}$,  are vector fields in $\r^4$ which depend only on $v$. Also, using \eqref{eq:value-a-b} we can write
$$
\alpha_{1,2}=\frac{1}{\varepsilon}(\lambda\,B\pm\varepsilon\,|\nu|\, \sqrt{\lambda\,B}).
$$ 
Now, as $|F|^2=1$ and using the equations~\eqref{eqquarta}, \eqref{eqprime},  \eqref{eqsegunda} given in the Proposition~\ref{position-vector}, we find that the position vector $F(u,v)$  and its derivatives  must satisfy the following relations:
\begin{equation}\label{eq:Fprocuct}
\begin{array}{lll}
\langle  F,F\rangle=1\,,& \langle F_u,F_u\rangle=\tilde{a}\,,& \langle F,F_u\rangle=0\,,\\
\langle F_u,F_{uu}\rangle=0\,,&\langle F_{uu},F_{uu}\rangle=D\,,&\langle F,F_{uu}\rangle=-\tilde{a},\\
\langle F_u,F_{uuu}\rangle=-D\,,&\langle F_{uu},F_{uuu}\rangle=0\,,&\langle F,F_{uuu}\rangle=0\,,\\
\langle F_{uuu},F_{uuu}\rangle=E,& \qquad &
\end{array}
\end{equation}
where 
$$
D=\tilde{a}\,\tilde{b}^2-3\,\tilde{a}^2\,,\qquad E=(\tilde{b}^2-2\tilde{a})\,D-\tilde{a}^3\,.
$$
Putting  $g_{ij}(v):=\langle g^i(v),g^j(v)\rangle$ and evaluating  the relations \eqref{eq:Fprocuct}  in $(0,v)$, it results that:
\begin{equation}\label{uno}
    g_{11}+g_{33}+2g_{13}=1\,,
\end{equation}
\begin{equation}\label{due}
    \alpha_1^2\,g_{22}+\alpha_2^2\,g_{44}+2\alpha_1\alpha_2\,g_{24}=\tilde{a}\,,
\end{equation}
\begin{equation}\label{tre}
    \alpha_1\,g_{12}+\alpha_2\,g_{14}+\alpha_1\,g_{23}+\alpha_2\,g_{34}=0\,,
\end{equation}
\begin{equation}\label{quatro}
    \alpha_1^3\,g_{12}+\alpha_1\alpha_2^2\,g_{23}+\alpha_1^2\alpha_2\,g_{14}+\alpha_2^3g_{34}=0\,,
\end{equation}
\begin{equation}\label{cinque}
    \alpha_1^4\,g_{11}+\alpha_2^4\,g_{33}+2\alpha_1^2\alpha_2^2\,g_{13}=D\,,
\end{equation}
\begin{equation}\label{sei}
    \alpha_1^2\,g_{11}+\alpha_2^2\,g_{33}+(\alpha_1^2+\alpha_2^2)\,g_{13}=\tilde{a}\,,
\end{equation}
\begin{equation}\label{sette}
    \alpha_1^4\,g_{22}+\alpha_1^3\alpha_2\,g_{24}+\alpha_1\alpha_2^3\,g_{24}+\alpha_2^4\,g_{44}=D\,,
\end{equation}
\begin{equation}\label{otto}
    \alpha_1^5\,g_{12}+\alpha_1^3\alpha_2^2\,g_{23}+\alpha_1^2\alpha_2^3\,g_{14}+\alpha_2^5\,g_{34}=0\,,
\end{equation}
\begin{equation}\label{nove}
    \alpha_1^3\,g_{12}+\alpha_1^3\,g_{23}+\alpha_2^3\,g_{14}+\alpha_2^3\,g_{34}=0\,,
\end{equation}
\begin{equation}\label{dieci}
    \alpha_1^6\,g_{22}+\alpha_2^6\,g_{44}+2\alpha_1^3\alpha_2^3\,g_{24}=E\,.
\end{equation}
From \eqref{tre}, \eqref{quatro}, \eqref{otto}, \eqref{nove}, it follows that 
$$
g_{12}=g_{14}=g_{23}=g_{34}=0\,.
$$
Also, from  \eqref{uno}, \eqref{cinque} and \eqref{sei}, we get
$$
g_{11}=\frac{D-2\tilde{a}\,\alpha_2^2+\alpha_2^4}{(\alpha_1^2-\alpha_2^2)^2}\,,\qquad g_{13}=0\,,\qquad
g_{33}=\frac{D-2\tilde{a}\,\alpha_1^2+\alpha_1^4}{(\alpha_1^2-\alpha_2^2)^2}.
$$
Moreover, using \eqref{due}, \eqref{sette} and \eqref{dieci}, we obtain
$$
g_{22}=\frac{E-2D\,\alpha_2^2+\tilde{a}\,\alpha_2^4}{\alpha_1^2\,(\alpha_1^2-\alpha_2^2)^2}\,,\qquad g_{24}=0\,,\qquad
g_{44}=\frac{E-2D\,\alpha_1^2+\tilde{a}\,\alpha_1^4}{\alpha_2^2\,(\alpha_1^2-\alpha_2^2)^2}\,.
$$
Finally, a long computation gives
$$
g_{11}=g_{22}=\frac{\lambda\,\varepsilon}{2B}\, \alpha_2\,,\qquad g_{33}=g_{44}=\frac{\lambda\,\varepsilon}{2B}\, \alpha_1
\,.
$$
\end{proof}
\begin{remark}\label{g13}
As $g_{13}=0$, from \eqref{uno} it results that $g_{11}+g_{33}=1$.
\end{remark}

\section{The characterization theorem of the helix surfaces in $\s^3_\varepsilon$}

We start this section proving a proposition that gives the conditions under which an immersion defines a helix surface in $\s^3_\varepsilon$. Before, we observe that if $F$ is the position vector of a helix surfaces in $\s^3_\varepsilon$, we have that
$$
{J_1}F(u,v)={X_1}_{|F(u,v)}=\varepsilon\,{E_1}_{|F(u,v)}=\varepsilon\,(F_u+\nu\,N)
$$ 
and, thus, using the equations~\eqref{eqquarta}--\eqref{eq:Fprocuct},  we obtain the following identities:
\begin{equation}\label{eq-fu-jf-main}
\begin{aligned} 
&\langle{J_1}F,F_u\rangle=\varepsilon^{-1}\,(1+\lambda\,\nu^2),\\&\langle{J_1}F,F_{uu}\rangle=0\,,\\
&\langle F_u,{J_1}F_{uu}\rangle=\tilde{a}\,[\varepsilon^{-1}\,(1+\lambda\,\nu^2)+\tilde{b}]:=I\,,\\
&\langle{J_1}F_u,F_{uuu}\rangle=0\,,\\
&\langle{J_1}F_u,F_{uu}\rangle+\langle{J_1}F,F_{uuu}\rangle=0\,,\\
&\langle{J_1}F_{uu},F_{uuu}\rangle+\langle{J_1}F_u,F_{uuuu}\rangle=0\,.
\end{aligned}
\end{equation}

\begin{proposition}\label{pro-viceversa}
Let $F:\Omega\to\s^3_\varepsilon$ be an immersion from an open set $\Omega\subset\r^2$, with local coordinates $(u,v)$.
Then, $F(\Omega)\subset \s^3_\varepsilon$ defines a helix  spacelike (respectively, timelike) surface of constant angle function $\nu$ and such that the projection of $E_1=\varepsilon^{-1}\,J_1F$ to the tangent space of $F(\Omega)\subset \s^3_\varepsilon$ is $F_u$, if and only if
\begin{equation}\label{viceversa1}
g_{\varepsilon}(F_u,F_u)=g_{\varepsilon}(E_1,F_u)=-(1+\lambda\,\nu^2)
\end{equation}
and
\begin{equation}\label{viceversa2}
g_{\varepsilon}(F_u,F_v)-g_{\varepsilon}(F_v,E_1)=0,
\end{equation}
where $\lambda=-1$ (respectively, $\lambda=1$).
\end{proposition}
\begin{proof}
Suppose that $F(\Omega)$ is a pseudo-Riemannian helix surface in $\s^3_\varepsilon$ of constant angle function $\nu$. With respect to the local coordinates $(u,v)$ defined in \eqref{eq:local-coordinates} and \eqref{solab}, we have that $F_u=(E_1)^T$ and, also, the equation~\eqref{gT} is fulfilled:
$$
 g_{\varepsilon}(F_u,F_u)=-(1+\lambda\,\nu^2)\,.
$$
In addition, from \eqref{eq-fu-jf-main}, we get
$$
 g_{\varepsilon}(E_1,F_u)=g_{\varepsilon}(F_u+\,\nu\,N,F_u)=g_{\varepsilon}(F_u,F_u)\,.
$$
Therefore, using \eqref{eq:local-coordinates}, we have that
$$
 g_{\varepsilon}(F_v,E_1)=g_{\varepsilon}(F_v,F_u+\,\nu\,N)=g_{\varepsilon}(F_u,F_v).
$$

For the converse, put
$$
\tilde{T}=F_v-\frac{g_{\varepsilon}(F_v,F_u)F_u}{g_{\varepsilon}(F_u,F_u)}\,.
$$
Then, if we denote by $N$ the unit normal vector field to the pseudo-Riemannian surface $F(\Omega)$ (i.e.  $g_{\varepsilon}(N,N)=\lambda$), we have that $\{F_u,\tilde{T},N\}$ is
an orthogonal basis of the tangent space of $\s^3_\varepsilon$ along the surface $F(\Omega)$.
Now, using  \eqref{viceversa1} and \eqref{viceversa2}, we get $g_{\varepsilon}(E_1,\tilde{T})=0$, thus $E_1=c_1\, F_u+c_2\, N$. Moreover,
using \eqref{viceversa1} and that $g_{\varepsilon}(E_1,F_u)=c_1\, g_{\varepsilon}(F_u,F_u)$, we conclude that $c_1=1$ (i.e. $F_u=(E_1)^T$).
Finally,
$$
-1=g_{\varepsilon}(E_1,E_1)=g_{\varepsilon}(F_u+c_2\, N,F_u+c_2\, N)=-(1+\lambda\,\nu^2)+\lambda\,c_ 2^2\,,
$$
which implies that $c_2^2=\nu^2$. 
Consequently,  up to the orientation of $N$, we obtain that
$$
g_{\varepsilon}(E_1,N)=g_{\varepsilon}(F_u+\nu\, N,N)=\lambda\,\nu=\cst
$$
and, thus, $F(\Omega)\subset \s^3_\varepsilon$ defines a pseudo-Riemannian helix surface.
\end{proof}

We are now in the right position to prove the main result of this paper.
\begin{theorem}\label{teo-principal}
Let $M$ be a helix surface in the Lorentzian Berger sphere $\s^3_\varepsilon$, with constant angle function $\nu$. Then, 
locally, 
the position vector of $M$ in $\r^4$, with respect to the local coordinates $(u,v)$ on $M$ defined in \eqref{eq:local-coordinates} and \eqref{solab}, is given by:
$$
F(u,v)=Q(v)\,\beta(u)\,,
$$
where
$$
\beta(u)=(\sqrt{g_{11}}\,\cos (\alpha_1 u),\lambda\,\sqrt{g_{11}}\,\sin (\alpha_1 u),\sqrt{g_{33}}\,\cos (\alpha_2 u),\lambda\,\sqrt{g_{33}}\,\sin (\alpha_2 u))
$$
is a twisted geodesic in the torus $\s^1(\sqrt{g_{11}})\times\s^1(\sqrt{g_{33}})\subset \s^3$, the constants $g_{11}$, $g_{33}$, $\alpha_1$, $\alpha_2$ are given in Corollary~\ref{cor-Fuv}, and $Q(v)$ is a $1$-parameter family of $4\times 4$ orthogonal matrices such that ${J_1}Q(v)=Q(v){J_1}$, with $\xi=\cst$ and
\begin{equation}\label{eq-alpha123}
  \cos^2(\xi_1(v)) \,\xi_{2}'(v)-\sin^2(\xi_1(v))\, \xi_{3}'(v)=0\,.
\end{equation}
Conversely, a parametrization $F(u,v)=Q(v)\,\beta(u)$, with $\beta(u)$ and $Q(v)$ as above, defines
	a helix surface in the Lorentzian Berger sphere $\s^3_\varepsilon$.
\end{theorem}
\begin{proof}
With respect to the local coordinates $(u,v)$ on $M$ defined in \eqref{eq:local-coordinates} and \eqref{solab},  the position vector of the  helix  surface in $\r^4$ is given by
$$
F(u,v)=\cos(\alpha_1 u)\,g^1(v)+\sin(\alpha_1 u)\,g^2(v)+\cos(\alpha_2 u)\,g^3(v)+\sin(\alpha_2 u)\,g^4(v)\,,
$$
where (see Corollary~\ref{cor-Fuv}) the vector fields  $\{g^i(v)\}$ are mutually orthogonal  and 
$$
||g^1(v)||=||g^2(v)||=\sqrt{g_{11}}=\text{constant}\,,
$$
$$
||g^3(v)||=||g^4(v)||=\sqrt{g_{33}}=\text{constant}\,.
$$
Putting $e_i(v):=g^i(v)/||g^i(v)||$, $i\in\{1,\dots,4\}$, we can write:
\begin{eqnarray}\label{eq:Fei}
F(u,v)=&\sqrt{g_{11}}\,(\cos (\alpha_1\,u)\,e_1(v)+\sin(\alpha_1\,u)\,e_2(v))\nonumber\\
&+\sqrt{g_{33}}\,(\cos (\alpha_2\,u)\,e_3(v)+\sin(\alpha_2\,u)\,e_4(v))\,.
\end{eqnarray}\\
Now, we will prove that 
$\bar{J}=\lambda\,(J_1)^{T}$, where $\bar{J}$ is the matrix with entries given by
$\bar{J}_{i,j}=\langle {J_1} e_i,e_j\rangle$, $i,j=1,\ldots,4$. 
Evaluating \eqref{eq-fu-jf-main} in  $(0,v)$, we get respectively:
\begin{equation}\label{eq1bis}  
    \alpha_1 g_{11}\langle{J_1}e_1,e_2\rangle+\alpha_2\,g_{33}\langle{J_1}e_3,e_4\rangle+\sqrt{g_{11}g_{33}}\,(\alpha_1\langle{J_1}e_3,e_2\rangle+\alpha_2\langle{J_1}e_1,e_4\rangle)
    =\varepsilon^{-1}\,(1+\lambda\,\nu^2),  
\end{equation}
\begin{equation}
     \langle{J_1}e_1,e_3\rangle=0\,,
\end{equation}
\begin{equation}\label{39}
\alpha_1^3\,g_{11}\langle{J_1}e_1,e_2\rangle+\alpha_2^3\,g_{33}\langle{J_1}e_3,e_4\rangle+\sqrt{g_{11}g_{33}}\,(\alpha_1\alpha_2^2\langle{J_1}e_3,e_2\rangle+\alpha_1^2\alpha_2\langle{J_1}e_1,e_4\rangle)=-I,
\end{equation}
\begin{equation}
    \langle{J_1}e_2,e_4\rangle=0\,,
\end{equation}
\begin{equation}\label{eq2bis}
    \alpha_1\langle{J_1}e_2,e_3\rangle+\alpha_2\langle{J_1}e_1,e_4\rangle=0\,,
\end{equation}
\begin{equation}\label{eq3bis}
    \alpha_2\langle{J_1} e_2,e_3\rangle+\alpha_1\langle{J_1}e_1,e_4\rangle=0\,.
\end{equation}
Note that to obtain the previous identities we have divided by $\alpha_1^2-\alpha_2^2=4 \varepsilon^{-1}\, \sqrt{\lambda\,\nu^2\,B^3 }$ which is, by the assumption on $\nu$, always different from zero.
From \eqref{eq2bis} and \eqref{eq3bis}, taking into account that $\alpha_1^2-\alpha_2^2\neq 0$, it results that
\begin{equation}\label{eq4bis}
     \langle{J_1}e_3,e_2\rangle=0\,,\qquad \langle{J_1}e_1,e_4\rangle=0\,.
\end{equation}
Consequently,
$$
|\langle{J_1}e_1,e_2\rangle|=1=|\langle{J_1}e_3,e_4\rangle|.
$$
Substituting \eqref{eq4bis} in \eqref{eq1bis} and \eqref{39}, we have the system
$$
\left\{\begin{aligned}
&\alpha_1 g_{11}\langle{J_1}e_1,e_2\rangle+\alpha_2\,g_{33}\langle{J_1}e_3,e_4\rangle=\varepsilon^{-1}\,(1+\lambda\,\nu^2),\\
 & \alpha_1^3 g_{11}\langle{J_1}e_1,e_2\rangle+\alpha_2^3\,g_{33}\langle{J_1}e_3,e_4\rangle=-I\,,
\end{aligned}
\right.
$$
a solution of which is
$$
\langle{J_1}e_1,e_2\rangle=\frac{\varepsilon\, I+\alpha_2^2\,(1+\lambda\,\nu^2)}{\varepsilon\, g_{11}\,\alpha_1 (\alpha_2^2-\alpha_1^2)}\,,\qquad \langle{J_1}e_3,e_4\rangle=-\frac{\varepsilon I+\alpha_1^2\,(1+\lambda\,\nu^2)}{\varepsilon\, g_{33}\,\alpha_2(\alpha_2^2-\alpha_1^2)}\,.
$$ 
Besides, since
$$
g_{11}\,g_{33}=\frac{1+\lambda\,\nu^2}{4B}\,,\qquad \alpha_1\,\alpha_2=\frac{B}{\varepsilon^2}\,(1+\lambda\,\nu^2)\,,
$$
we get
$
\langle{J_1}e_1,e_2\rangle\langle{J_1}e_3,e_4\rangle=1\,.
$
Moreover,  it's easy to check that 
$\langle {J_1}e_1,e_2\rangle=\lambda$.
Consequently,
$\langle {J_1}e_3,e_4\rangle=\lambda$ and we have proved that $\bar{J}=\lambda\,({J_1})^{T}$. \\

Then, if we fix the orthonormal basis  of $\r^4$ given by 
$$
\tilde{E_1}=(1,0,0,0)\,,\quad \tilde{E_2}=(0,\lambda,0,0)\,,\quad \tilde{E_3}=(0,0,1,0)\,,\quad\tilde{ E_4}=(0,0,0,\lambda)\,,
$$ 
there must exists a $1$-parameter family of orthogonal matrices $Q(v)\in \mathrm{O}(4)$, with ${J_1}Q(v)=Q(v){J_1}$, such that $e_i(v)=Q(v)\tilde{E_i}$.  So, from \eqref{eq:Fei} we have
$$
F(u,v)=Q(v)\beta(u)\,,
$$ 
where the curve 
$$
\beta(u)=(\sqrt{g_{11}}\,\cos (\alpha_1 u),\lambda\,\sqrt{g_{11}}\,\sin (\alpha_1 u),\sqrt{g_{33}}\,\cos (\alpha_2 u),\lambda\,\sqrt{g_{33}}\,\sin (\alpha_2 u))\,,
$$
is a twisted geodesic of the torus $\s^1(\sqrt{g_{11}})\times\s^1(\sqrt{g_{33}})$, which is contained in the sphere  $\s^3$ (see Remark~\ref{g13}).\\

Now, we consider the description of the $1$-parameter family $Q(v)$ given in \eqref{eq-descrizione-A} (see~\cite{MO}), that makes use of the four functions $\xi_1(v),\xi_2(v),\xi_3(v)$ and $\xi(v)$. From \eqref{eq:local-coordinates} and \eqref{solab}, it results that $\langle F_v, F_v\rangle=\lambda+\nu^2$ and, therefore,
\begin{equation}\label{eq-fv-fv-sin-theta-d-u}
\frac{\partial}{\partial u}\langle F_v, F_v\rangle_{| u=0}=0\,.
\end{equation}

Moreover, if we denote by ${\mathbf q_1},{\mathbf q_2},{\mathbf q_3},{\mathbf q_4}$ the colons of $Q(v)$, equation \eqref{eq-fv-fv-sin-theta-d-u} implies that
\begin{equation}\label{sistem-c23-c24}
\langle {\mathbf q_2}',{\mathbf q_3}'\rangle=0,\qquad
\langle {\mathbf q_2}',{\mathbf q_4}'\rangle=0\,,
\end{equation}
where $'$ denotes the derivative with respect to $v$.
Substituting in \eqref{sistem-c23-c24} the expressions of the ${\mathbf q_i}$'s as functions of $\xi_1(v),\xi_2(v),\xi_3(v)$ and $\xi(v)$, we obtain
$$
\begin{cases}
\xi'\, h(v)=0,\\
\xi'\, k(v)=0,
\end{cases}
$$
where $h(v)$ and $k(v)$ are two functions such that
$$
h^2+k^2=4 (\xi_1')^2+\sin^2(2\xi_1)\, (-\xi'+\xi_2'+\xi_3')^2\,.
$$
Consequently, we have two possibilities:
\begin{itemize}
\item[(i)] $\xi=\cst$;
\item[] or
\item[(ii)] $4 (\xi_1')^2+\sin^2(2\xi_1)\, (-\xi'+\xi_2'+\xi_3')^2=0$.
\end{itemize}

We will show that case (ii) cannot occurs, more precisely we will show that if (ii) happens  than the parametrization $F(u,v)=Q(v)\beta(u)$ defines a timelike Hopf tube, that is the vector field $E_1$ is tangent to the surface. To this end,  we write the unit normal vector field $N$ to the parametrization
$F(u,v)=Q(v) \beta(u)$ as:
$$
N=\frac{N_1 E_1+N_2E_2+N_3 E_3}{\sqrt{|-N_1^2+N_2^2+N_3^2|}}\,,
$$
where
$$
\begin{cases}
N_1=g_{\varepsilon}(F_u,E_3)g_{\varepsilon}(F_v,E_2)-g_{\varepsilon}(F_u,E_2)g_{\varepsilon}(F_v,E_3),\\
N_2=g_{\varepsilon}(F_u,E_1)g_{\varepsilon}(F_v,E_3)-g_{\varepsilon}(F_u,E_3)g_{\varepsilon}(F_v,E_1),\\
N_3=g_{\varepsilon}(F_u,E_2)g_{\varepsilon}(F_v,E_1)-g_{\varepsilon}(F_u,E_1)g_{\varepsilon}(F_v,E_2)\,.
\end{cases}
$$

Now case (ii) occurs if and only if $\xi_1=\cst=0$, or if $\xi_1=\cst\neq 0$ and $-\xi'+\xi_2'+\xi_3'=0$. In both cases $N_1=0$ and this implies that $g_{\varepsilon}(N,J_1F)=\varepsilon\,g_{\varepsilon}(N,E_1)=0$, i.e. the timelike Hopf vector field $E_1$ is tangent to the surface, which is a timelike Hopf tube. Thus, we have proved that $\xi=\cst$. 
Finally, in this case, \eqref{viceversa2} is equivalent to
$$
\lambda\,\nu\,\sqrt{\lambda\,B}\,[\cos^2(\xi_1(v)) \,\xi_{2}'(v)-\sin^2(\xi_1(v))\, \xi_{3}'(v)]=0
$$
and, as $\nu\neq 0$, we conclude that the  condition~\eqref{eq-alpha123} is satisfied.\\

The converse follows immediately from Proposition~\ref{pro-viceversa} since a direct
calculation shows that $$g_{\varepsilon}(F_u,F_u)=g_{\varepsilon}(E_1,F_u)=-(1+\lambda\,\nu^2)$$ which is \eqref{viceversa1},
while \eqref{eq-alpha123} is equivalent to \eqref{viceversa2}.
\end{proof}

\begin{corollary}\label{cor-2}
Let $M$ be a helix spacelike (respectively, timelike) surface in the Lorentzian Berger sphere $\s^3_\varepsilon$ with constant angle function $\nu$. Then, there exist local coordinates on $M$ such that the position vector of $M$ in $\r^4$  is
$$
F(s,v)=Q(v)\,\beta(s)\,,
$$
where
\begin{equation}\label{ppca}
\beta(s)=\frac{1}{\sqrt{1+d^2}}\Big(d\,\cos \Big(\frac{s}{d}\Big),\lambda\,d\,\sin \Big(\frac{s}{d}\Big),\cos (d\,s),\lambda\,\sin (d\,s)\Big)
\end{equation}
is a twisted geodesic in the torus $\s^1(\frac{d}{\sqrt{1+d^2}})\times\s^1(\frac{1}{\sqrt{1+d^2}})\subset \s^3$ parametrized by arc length, whose slope is given by:
$$
d=\frac{\sqrt{\lambda\,B}-\varepsilon\,|\nu|}{\sqrt{\lambda+\nu^2}}\in (0,1),
$$
where $\lambda=-1$ (respectively, $\lambda=1$). In addition, $Q(v)=Q(\xi,\xi_1,\xi_2,\xi_3)(v)$ is a $1$-parameter family of $4\times 4$ orthogonal matrices commuting with $J_1$, as described in \eqref{eq-descrizione-A}, with $\xi=\cst$ and
$$
  \cos^2(\xi_1(v)) \,\xi_{2}'(v)-\sin^2(\xi_1(v))\, \xi_{3}'(v)=0\,.
$$
Conversely, a parametrization $F(s,v)=Q(v)\,\beta(s)$, with $\beta(s)$ and $Q(v)$ as above, defines
 a helix surface in the Lorentzian Berger sphere $\s^3_\varepsilon$.
\end{corollary}
\begin{proof}
We consider the curve $\beta(u)$ given in the Theorem~\ref{teo-principal}. 
Since $\langle\beta'(u),\beta'(u)\rangle=\alpha_1\,\alpha_2$, considering $$d:=\sqrt{\frac{\alpha_2}{\alpha_1}}=\frac{\sqrt{\lambda\,B}-\varepsilon\,|\nu|}{\sqrt{\lambda+\nu^2}},$$ from the equations~\eqref{G11}, \eqref{G33} and, also, taking into account the Remark~\ref{g13}, we get
$$g_{11}=\frac{d^2}{1+d^2},\qquad g_{33}=\frac{1}{1+d^2}$$ and, also,  we observe that $0<d<1$. Therefore,
we can consider the  arc length reparameterization  of the curve $\beta$ given by:
$$
\beta(s)=\frac{1}{\sqrt{1+d^2}}\Big(d\,\cos \Big(\frac{s}{d}\Big),\lambda\,d\,\sin \Big(\frac{s}{d}\Big),\cos (d\,s),\lambda\,\sin (d\,s)\Big).
$$
Finally, we observe that $d$ represents the slope of the geodesic $\beta$.
\end{proof}
\begin{remark}
The curve $\beta:\r\to\s^3$ parametrized by~\eqref{ppca} is a spherical helix in $\s^3$ with constant geodesic curvature and torsion given by: $$\kappa_g=\frac{1-d^2}{d}=\frac{2\varepsilon\,|\nu|}{\sqrt{\lambda+\nu^2}},\qquad |\tau_g|=1.$$ 
\end{remark}
\begin{proposition}
The curve $\beta:\r\to\s_{\varepsilon}^3$ parametrized by~\eqref{ppca}, that is used in the Corollary~\ref{cor-2} to characterize a constant angle spacelike (respectively, timelike) surface $M$,  is a spacelike (respectively, timelike)  general helix\footnote{A non-null curve 
$\beta$ in a Lorentzian manifold $(N,h)$ is called a {\it general helix} if there exists a Killing vector
field $V$ with constant length along $\beta$ and such that the angle function between $V$ and $\beta'$ (i.e. $h(\beta',V)/\|\beta'\|_h$)
is a non-zero constant along $\beta$. We say that $V$ is an axis of the general helix $\beta$.} in $\s_\varepsilon^3$ with axis $E_1$, i.e. it has constant angle with the fibers of the Hopf fibration.
\end{proposition}
\begin{proof}
Firstly, we observe that, as $\beta:\r\to\r^4$ is parametrized by arc length, then  $$g_\varepsilon(\beta',\beta')=1-(1+\varepsilon^2)\langle\beta',J_1\beta\rangle^2=-\frac{\varepsilon^2}{B},$$ where the constant $B$ is negative (respectively, positive) if $M$ spacelike (respectively, timelike). Therefore, it results that $\beta:\r\to\s_{\varepsilon}^3$ is a spacelike (respectively, timelike) curve. Moreover, since $$\|\beta'\|_{\varepsilon}=\varepsilon\,\sqrt{\frac{\lambda}{B}},\qquad g_\varepsilon(\beta',E_1)=-\varepsilon\frac{\sqrt{\lambda\,B\,(\lambda+\nu^2)}}{B},$$ then the angle function between $\beta'$ and the Hopf vector field, given by
\begin{equation}\label{beta1}
\frac{g_\varepsilon(\beta',E_1)}{\|\beta'\|_{\varepsilon}}=-\lambda\,\sqrt{\lambda+\nu^2},
\end{equation} is constant. So, the curve $\beta$ is a general helix in $\s_{\varepsilon}^3$.
\end{proof}

\begin{corollary}
Let $M$ be a helix surface in the Lorentzian Berger sphere $\s^3_\varepsilon$, parametrized by $F(s,v)=Q(v)\beta(s)$. Then, the hyperbolic angle between its normal vector field $N$ and $E_1$ is the same that the general helix $\beta$ makes with its axis $E_1$.
\end{corollary}
\begin{proof}
Let $M$ be a spacelike surface  in $\s^3_\varepsilon$, with constant angle function $\nu$. Then, there exists a unique $\vartheta\geq 0$ (up to the orientation of $N$) such that $\nu=\cosh\vartheta$, where $\vartheta$ is called the {\it hyperbolic angle} between $N$ and $E_1$. We remember that, in this case, $\vartheta\neq 0$ since the horizontal distribution of the Hopf map
is not integrable. The conclusion follows from $\eqref{beta1}$,  observing that $-\lambda\,\sqrt{\lambda+\nu^2}=\sinh\vartheta$ and that $\beta'$ is a spacelike vector field.\\

If $M$ is a helix timelike surface, then the hyperbolic angle satisfies $\nu=\sinh\vartheta$, where $\vartheta\neq 0$ since we are not considering Hopf tubes. Consequently, $-\lambda\,\sqrt{\lambda+\nu^2}=-\cosh\vartheta$ and from \eqref{beta1}, up to the orientation of the general helix $\beta$,  we conclude that $\vartheta> 0$ is the hyperbolic angle that it makes with the Hopf vector field $E_1$.
\end{proof}

In the following, we will construct some explicit examples of helix surfaces
   in $\s^3_\varepsilon$.
\begin{example}
Taking
$$
\xi=\pi/2\,,\qquad \xi_1=\pi/4\,,\qquad \xi_2(v)=\xi_3(v)\,,
$$
 from \eqref{eq-descrizione-A} we obtain the  following $1$-parameter family of matrices $Q(v)$ that satisfies the conditions of the Corollary~\ref{cor-2}:
$$
Q(v)=\frac{1}{\sqrt{2}}\begin{pmatrix}
  \cos \xi_2(v) & -\sin \xi_2(v) & \cos \xi_2(v) & -\sin \xi_2(v) \\
 \sin \xi_2(v)& \cos \xi_2(v) & \sin \xi_2(v) & \cos \xi_2(v) \\
 -\cos \xi_2(v) & -\sin \xi_2(v) & \cos \xi_2(v) & \sin \xi_2(v) \\
 \sin\xi_2(v) & -\cos \xi_2(v) & -\sin \xi_2(v) & \cos \xi_2(v)
\end{pmatrix}\,.
$$

In the Figure~\ref{scelta1}
we have plotted the stereographic projection in $\r^3$ of surfaces, with constant angle function $\nu=4$, parametrized by $F(s,v)=Q(v)\,\beta(s)$, with $s\in (-4\pi,4\pi)$ and $v\in (-2\pi,2\pi)$, in the case $\xi_2(v)=v$.\\

\begin{figure}[h]
\begin{center}
\includegraphics[width=0.22\linewidth]{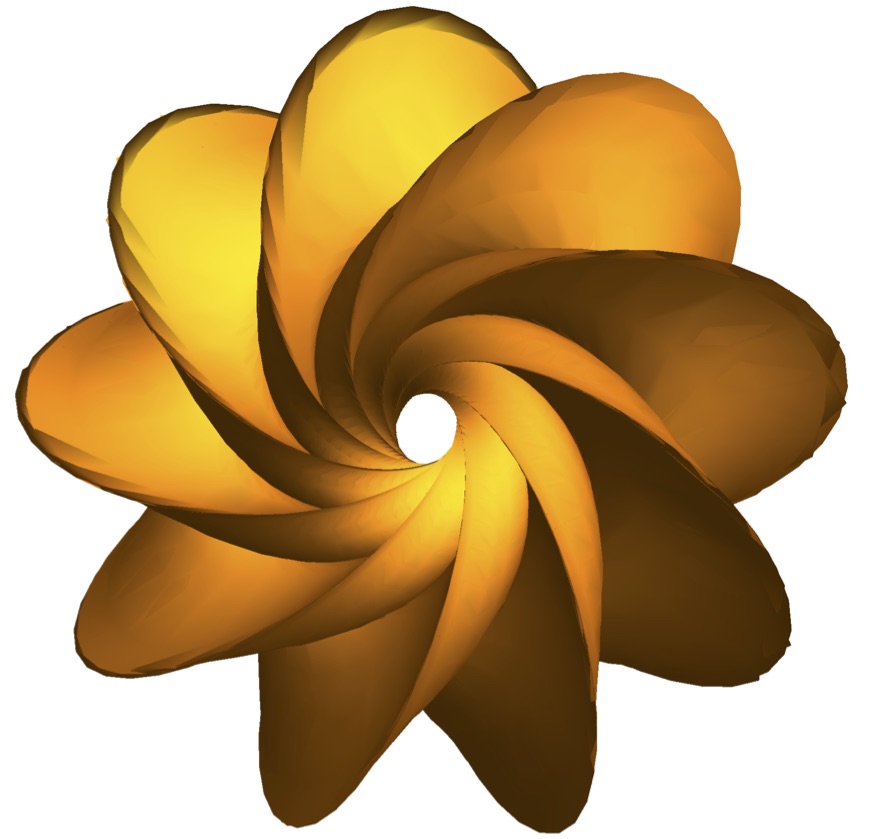}\qquad \quad\qquad
\includegraphics[width=0.22\linewidth]{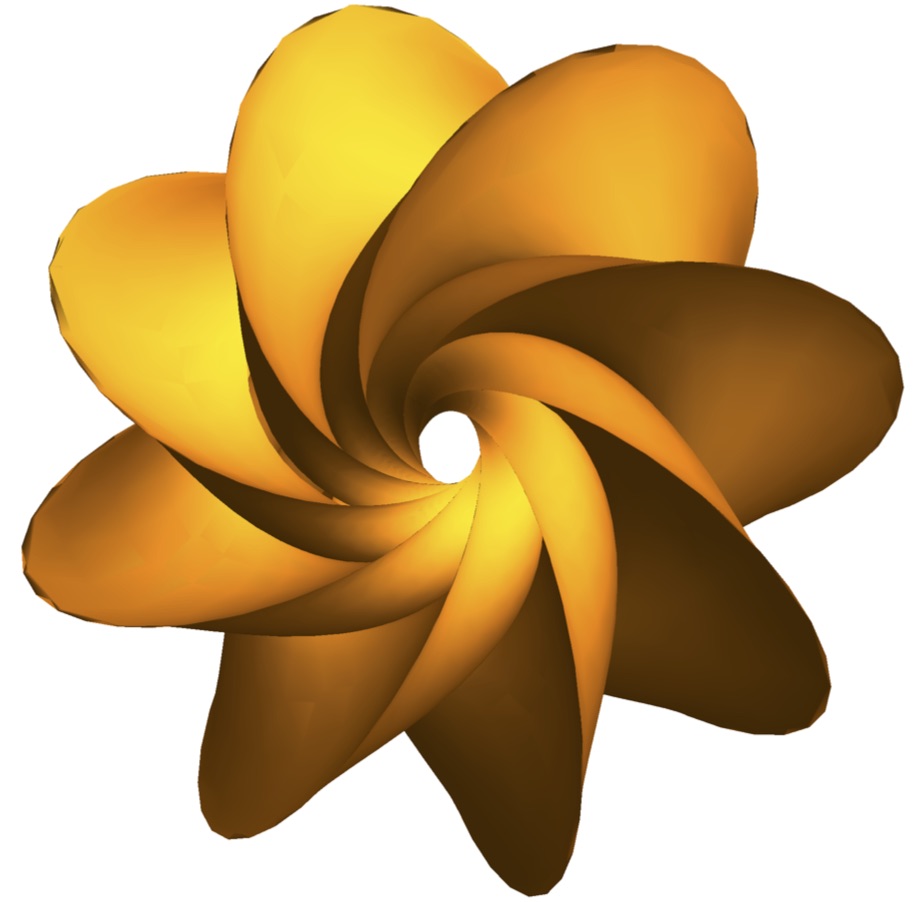}
\end{center}
\caption{\it Stereographic projection in $\r^3$ of the helix spacelike and timelike surface (respectively) of constant angle function $\nu=4$, obtained for $\varepsilon =2$.}	
\label{scelta1}
\end{figure}

However, in the Figure~\ref{scelta2}
we can visualize the stereographic projection in $\r^3$ of surfaces, with constant angle function $\nu=2$, parametrized by $F(s,v)=Q(v)\,\beta(s)$, with $s\in (-2\pi,2\pi)$ and $v\in (-2,2)$, that are obtained choosing $\xi_2(v)=e^v$.
\begin{figure}[!h]
\begin{center}
\includegraphics[width=0.22\linewidth]{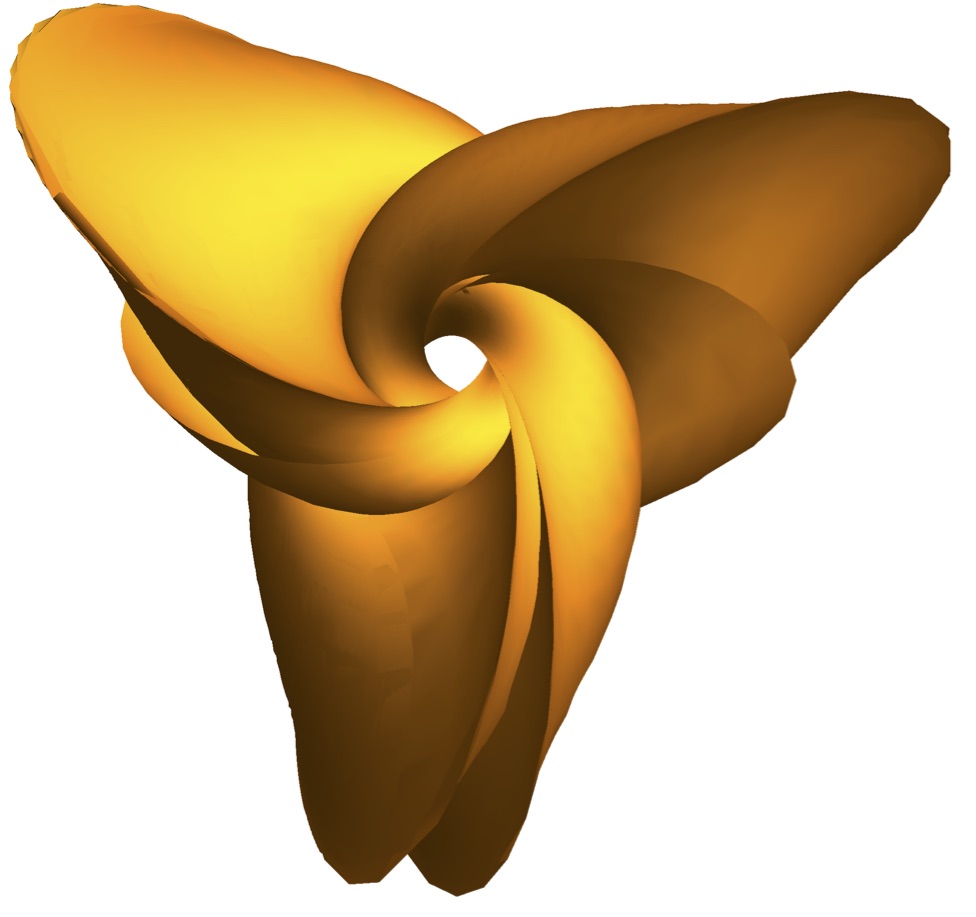}\qquad \qquad\qquad\quad
\includegraphics[width=0.09\linewidth]{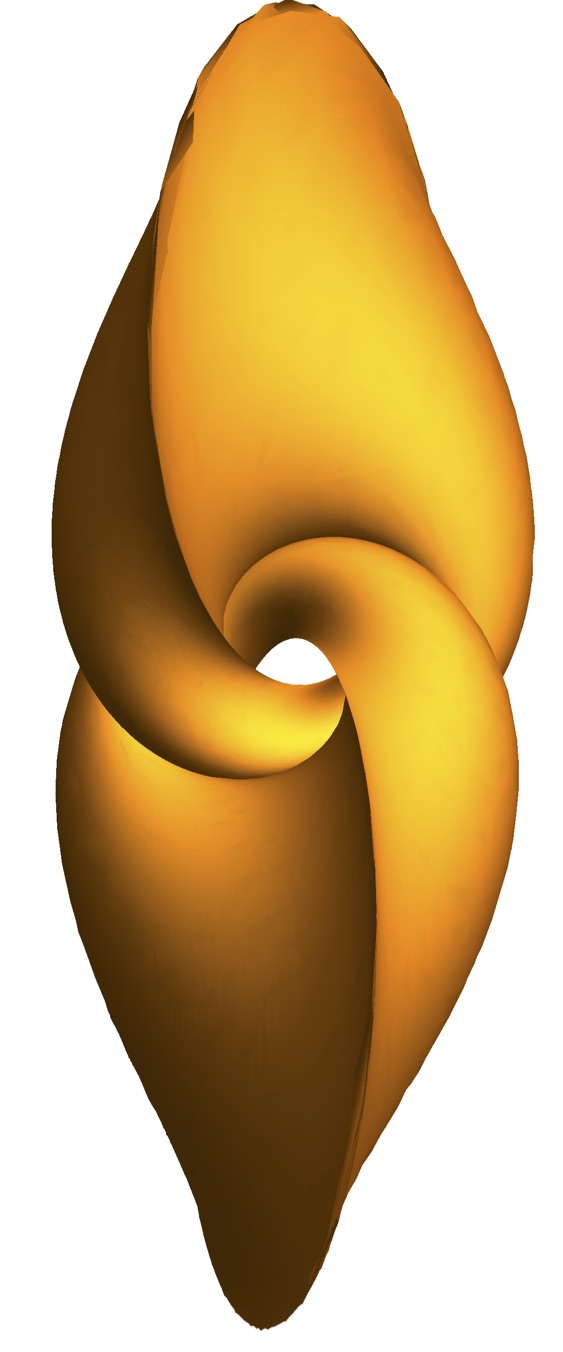}
\end{center}
\caption{\it Stereographic projection in $\r^3$ of the helix spacelike and timelike surface (respectively) of constant angle function $\nu=2$, obtained for $\varepsilon =1$.}	
\label{scelta2}
\end{figure}
\end{example}

\begin{example}
We consider a constant angle surface $F(u,v)=Q(v)\beta(u)$, with local coordinates $(u,v)$ defined as in \eqref{eq:local-coordinates} and \eqref{solab}. In the proof of Theorem~\ref{teo-principal} we see that, denoting by ${\mathbf q_1},{\mathbf q_2},{\mathbf q_3},{\mathbf q_4}$ the four colons of the matrix $Q(v)=Q(\xi,\xi_1,\xi_2,\xi_3)(v)$, it results that (see \eqref{sistem-c23-c24}):
$$
\langle {\mathbf q_2}',{\mathbf q_3}'\rangle=0,\qquad
\langle {\mathbf q_2}',{\mathbf q_4}'\rangle=0
$$
and, therefore, $\xi=\cst$ and 
\begin{equation}\label{xi-uno}
\cos^2(\xi_1(v)) \,\xi_{2}'(v)-\sin^2(\xi_1(v))\, \xi_{3}'(v)=0\,.
\end{equation}
Also, it's easy to check that
$$\langle F_v,F_v\rangle=\langle {\mathbf q_i}',{\mathbf q_i}'\rangle,\qquad i=1,\dots,4.$$ Thus, we get
\begin{equation}\label{xi-due}
\lambda+\nu^2=\xi_1'(v)^2+\cos^2(\xi_1(v))\,\xi_2'(v)^2+\sin^2(\xi_1(v))\,\xi_3'(v)^2.
\end{equation}
If we suppose $\xi_1(v)=\cst$ with $\cos(\xi_1)\neq 0$ and $\sin(\xi_1)\neq 0$, from \eqref{xi-uno} and \eqref{xi-due} we get
$$\xi_2(v)=\tan\xi_1\,\sqrt{\lambda+\nu^2}\,v+d_2,\qquad \xi_3(v)=\cot\xi_1\,\sqrt{\lambda+\nu^2}\,v+d_3,$$
with $d_i\in\r$, $i=2,3$. In particular, choosing $d_2=0=d_3$ and $\xi_1=1/\sqrt{1+d^2}$ (where $d\in (0,1)$ is the constant given in the Corollary~\ref{cor-2}), we obtain $$\xi_2(v)=d^{-1}\,\sqrt{\lambda+\nu^2}\,v,\qquad \xi_3(v)=d\,\sqrt{\lambda+\nu^2}\,v$$ and the corresponding immersion $F(s,v)=Q(v) \beta(s)$ of a helix surface into the Lorentzian Berger sphere depends only of $\nu$ and $\lambda$.
By using this parametrization composed with the stereographic
projection in $\r^3$,  we can visualize two examples of helix surfaces in the Lorentzian Berger sphere $\s_1^3$ given in the Figures~\ref{scelta3} and \ref{scelta3-bis}.\\
\begin{figure}[!h]
\begin{center}
\includegraphics[width=0.24\linewidth]{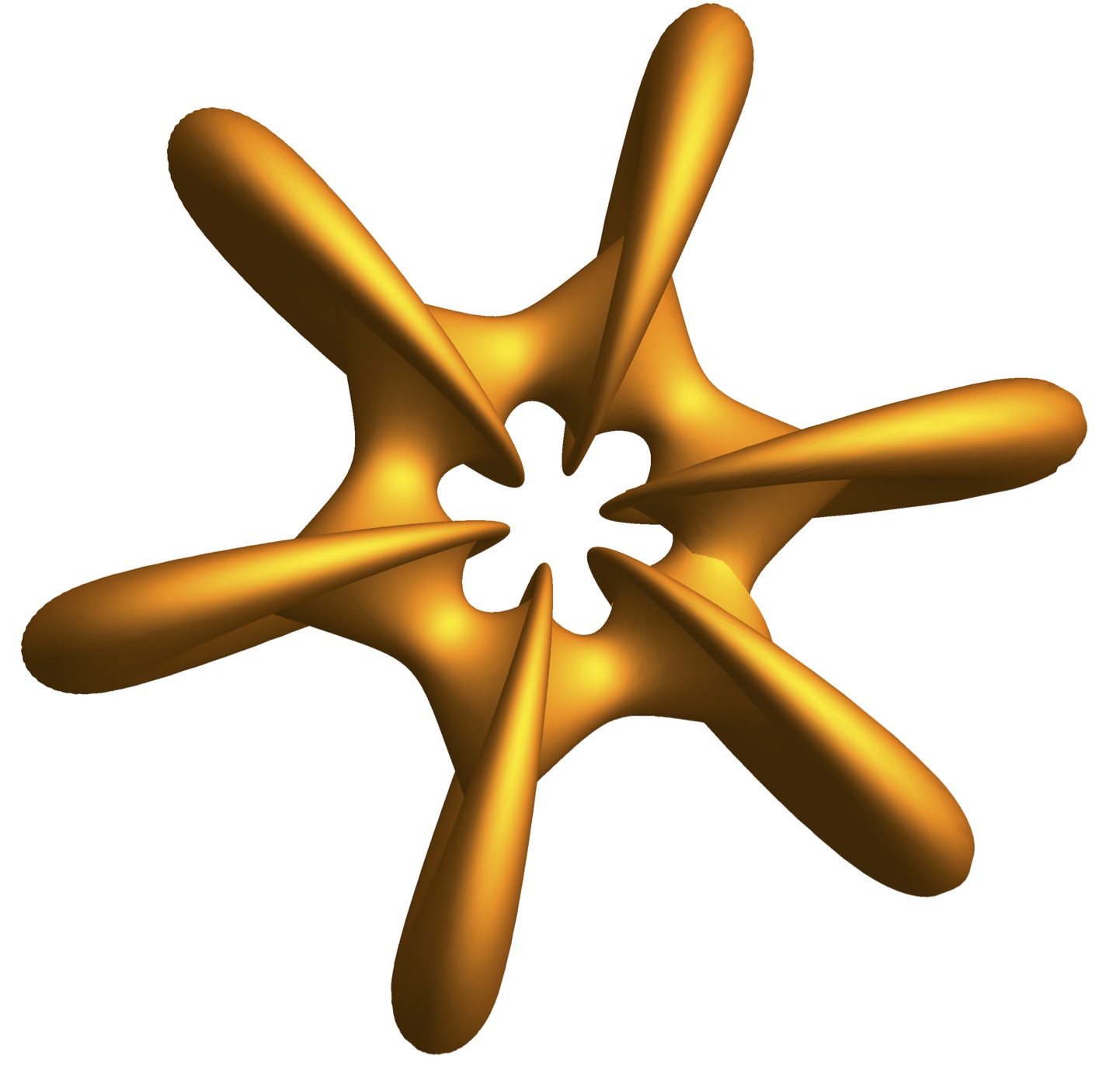}\qquad \qquad\qquad
\includegraphics[width=0.26\linewidth]{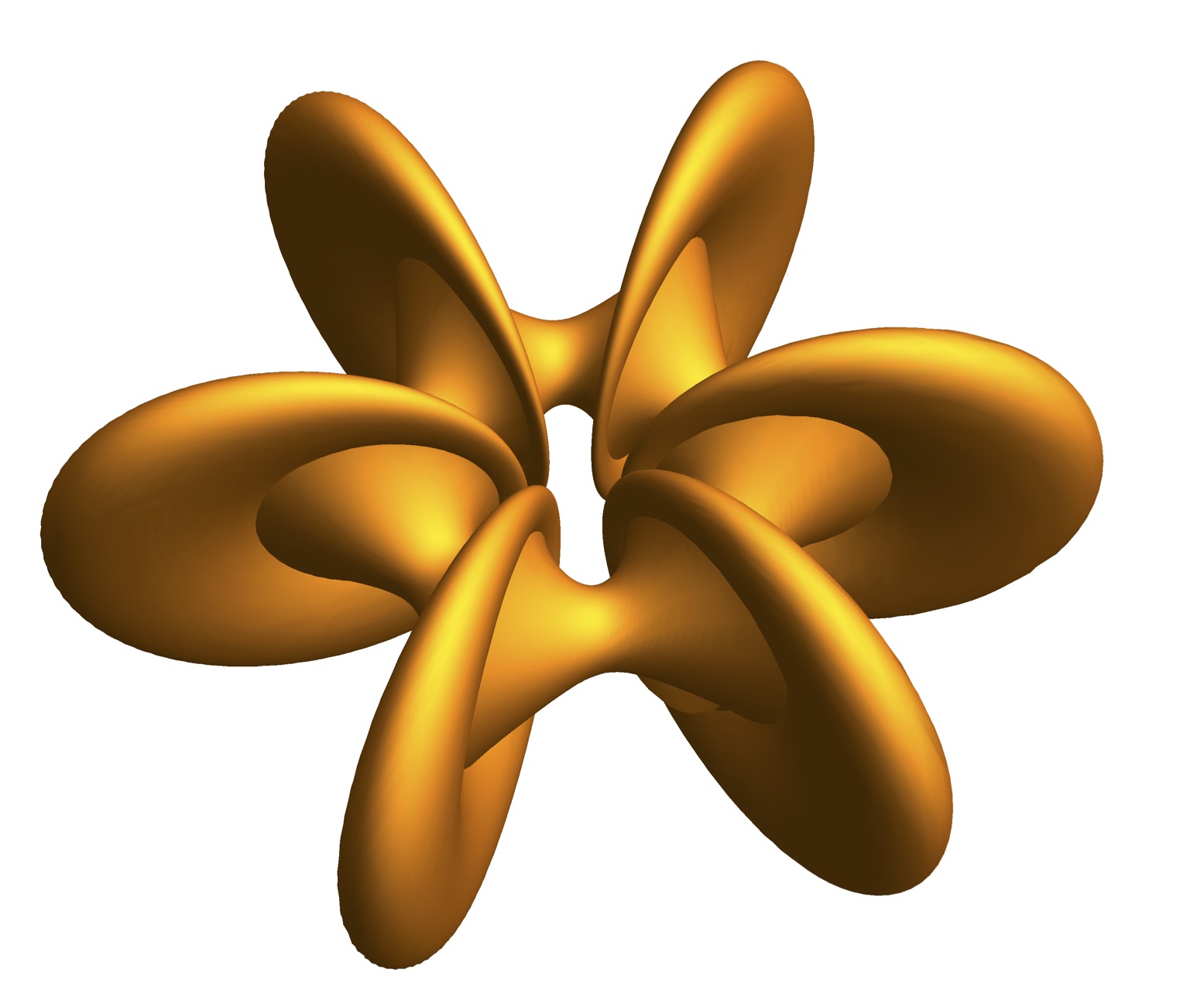}
\end{center}
\caption{\it Stereographic projection in $\r^3$ of the helix spacelike  surface of constant angle function $\nu=\sqrt{5}$, obtained for $\varepsilon =1$.}	
\label{scelta3}
\end{figure}

\begin{figure}[!h]
\begin{center}
\includegraphics[width=0.21\linewidth]{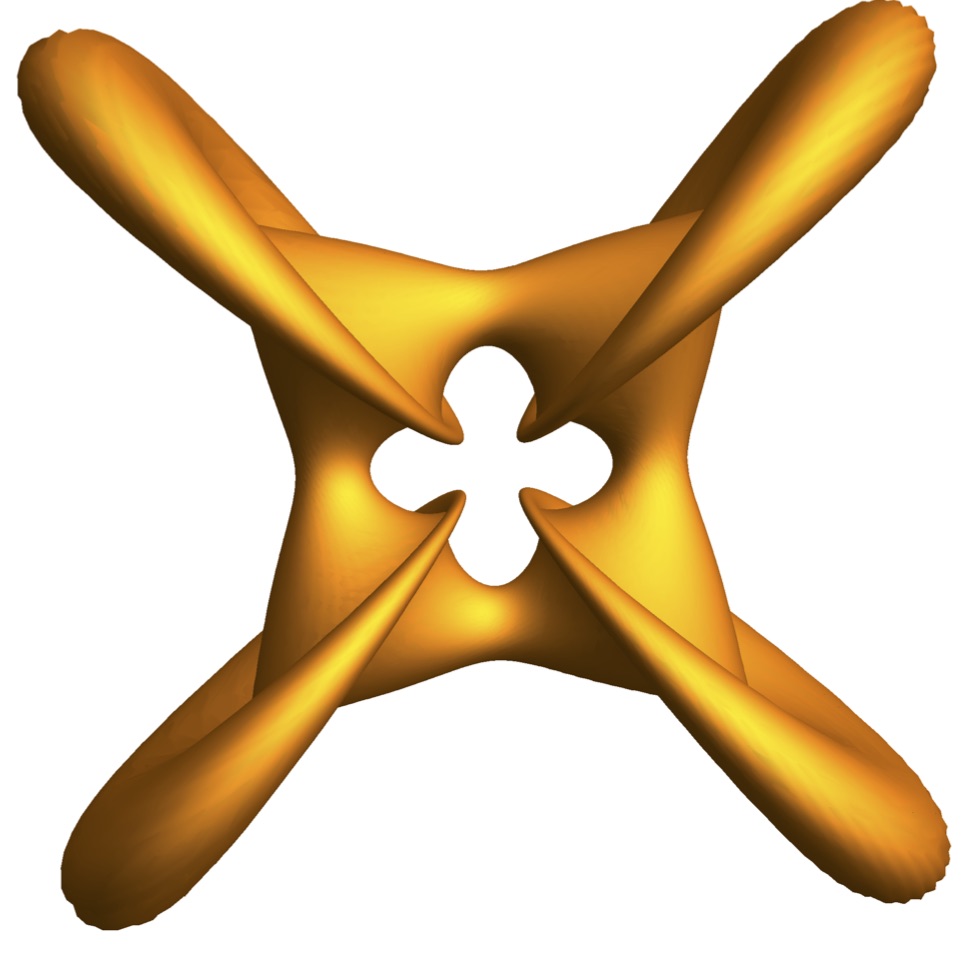}\qquad \qquad\qquad
\includegraphics[width=0.24\linewidth]{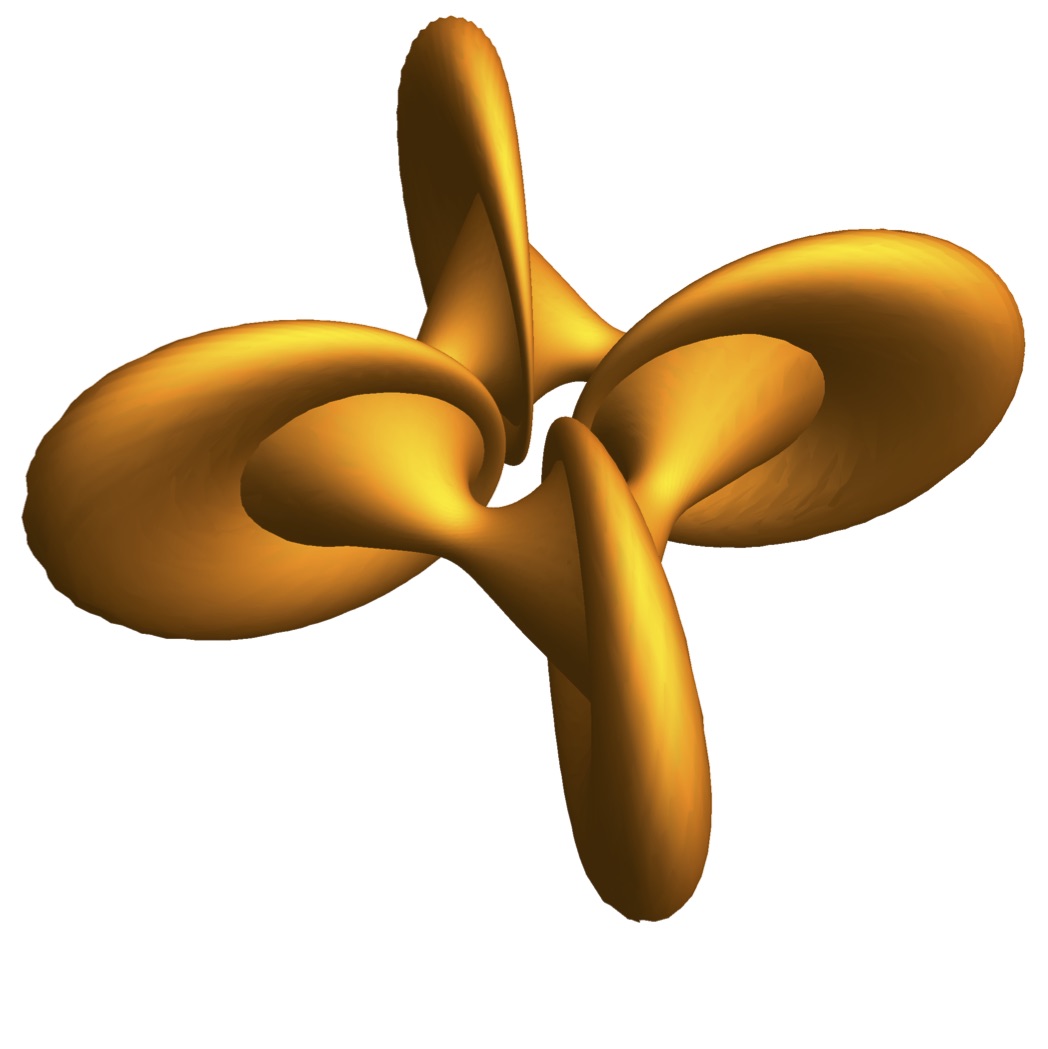}
\end{center}
\caption{\it Stereographic projection in $\r^3$ of the helix timelike surface of constant angle function $\nu=2$ in the Lorentzian Berger sphere $\s_1^3$.}	
\label{scelta3-bis}
\end{figure}
\end{example}

\end{document}